\documentclass[a4paper,11pt]{article}
\usepackage{amsfonts}
\usepackage{amsmath}
\usepackage{amsbsy}
\usepackage{amsxtra}
\usepackage{latexsym}
\usepackage{amssymb}
\usepackage{url}
\usepackage{enumerate}
\usepackage{amscd}
\usepackage[active]{srcltx}
\usepackage{cite}
\newtheorem{theorem}{Theorem}
\newtheorem{lemma}{Lemma}

\newtheorem{proposition}{Proposition}
\newtheorem{remark}{Remark}

\newtheorem{problem}{Problem}

\def\argmin{\operatorname{argmin}}

\DeclareMathOperator{\diag}{diag}
\DeclareMathOperator{\sgn}{sgn}

\newcommand{\lng}{\langle}
\newcommand{\rng}{\rangle}

\newcommand{\R}{\mathbb R}

\newcommand{\f}{\frac}
\newcommand{\ds}{\displaystyle}

\newenvironment{proof}{{\noindent\bf Proof.}}{\hfill$\Box$\\}

\begin{document}

\title{Projection onto simplicial cones by a semi-smooth Newton method
\thanks{{\it 1991 A M S Subject Classification.} Primary 90C33;
Secondary 15A48, {\it Key words and phrases.} Metric projection onto simplicial cones}}

\author{ O. P. Ferreira\thanks{IME/UFG, Campus II- Caixa Postal 131, 
Goi\^ania, GO, 74001-970, Brazil (e-mail:{\tt
orizon@ufg.br}).  The author was supported in part by
FAPEG, CNPq Grants 471815/2012-8,  303732/2011-3 and PRONEX--Optimization(FAPERJ/CNPq).}  
\and
S. Z. N\'emeth \thanks{School of Mathematics, The University of Birmingham, The Watson Building, 
Edgbaston, Birmingham B15 2TT, United Kingdom
(e-mail:{\tt nemeths@for.mat.bham.ac.uk}). The author was supported in part by 
the Hungarian Research Grant OTKA 60480.} }

\maketitle

\begin{abstract}
By using Moreau's decomposition theorem for projecting onto cones, the problem of projecting onto a simplicial cone is reduced to finding the unique solution of a 
nonsmooth system of equations. It is shown that a semi-smooth Newton method applied to the  system of equations associated  to the  problem of projecting onto a simplicial
cone is always  well defined, and  the generated sequence is   bounded  for any starting point and  under a somewhat restrictive assumption it is finite. Besides, under a 
mild assumption on the simplicial cone,  the generated sequence converges linearly to the  solution of the associated system of equations. 

\end{abstract}

\section{Introduction}

The interest in the subject of projection arises in several situations,   having a wide range of applications in pure and applied mathematics such as  Convex Analysis 
(see e.g. \cite{HiriartLemarecal1}),   Optimization (see  e.g. \cite{BuschkeBorwein96}, \cite{censor07}, \cite{censor01},  \cite{Frick1997}, \cite{scolnik08}, \cite{ujvari2007projection}), Numerical Linear Algebra 
(see e.g. \cite{Stewart77}), Statistics  (see e.g. \cite{BerkMarcus96}, \cite{Dykstra83}, \cite{Xiaomi1998}), Computer Graphics (see e.g. \cite{Fol90} ) and  Ordered 
Vector Spaces (see e.g. \cite{AbbasNemeth2012}, \cite{IsacNem86}, \cite{IsacNem92}, \cite{NemethNemeth2009}, \cite{Nemeth20091}, \cite{Nemeth2010-2}).   More specifically, the projection onto a polyhedral 
cone,  which has as a special case the projection onto a simplicial one, is a problem of high impact on scientific community\footnote{see the popularity of the 
Wikimization page Projection on Polyhedral Cone at http://www.convexoptimization.com/wikimization/index.php/Special:Popularpages}.  The geometric  nature of this problem 
makes it particularly interesting and important  in many areas of science and technology such as   Statistics~(see e.g. \cite{Xiaomi1998}), 
Computation~(see e.g. \cite{Huynh1992}), Optimization (see  e.g.\cite{Morillas2005}, \cite{ujvari2007projection}) and Ordered Vector Spaces (see e.g. \cite{NemethNemeth2009}). 

In this paper we particularize the Moreau's decomposition theorem for simplicial cones. 
This leads to an equivalence between the problem of projecting a point onto a simplicial cone 
and the one of finding the unique solution of a nonsmooth system of equations.   We apply  a semi-smooth Newton method for finding a unique solution of  the  obtained 
associated system.  We show that the method  is always  well defined and  the generated sequence is   bounded  for any starting point, and   under a somewhat
restrictive assumption it is finite. Besides, under a mild assumption on the simplicial cone,  the generated sequence converges linearly to the  solution of the 
associated  system of equations.  It is worth pointing out that a similar approach has been considered by Mangasarian in \cite{Mangasarian2009} for finding solutions of 
NP-hard absolute value equations.

\section{Preliminaries} \label{sec:int.1}

Consider $\R^m$ endowed with an orthogonal coordinate system and let 
$\lng\cdot,\cdot\rng$ be the canonical scalar product defined by it. Denote by $\|\cdot\|$ be
the norm generated by $\lng\cdot,\cdot\rng$.  If $a\in\R$ and $x=(x^1,\dots,x^m)\in\R^m$, then denote $a^+:=\max\{a,0\}$, $a^-:=\max\{-a,0\}$ and 
$$ 
x^+:=\left((x^1)^+,\dots,(x^m)^+\right).
$$
For $x\in \R^m$, the vector $\sgn(x)$ will denote a vector with components equal to $1$, $0$ or $-1$ depending on whether the corresponding component of the vector $x$ is positive, zero or negative. We will call a closed set $K\subset\R^m$ a \emph{cone} if the following conditions hold:
\begin{enumerate}
	\item $\lambda x+\mu y\in K$ for any $\lambda,\mu\ge0$ and $x,y\in K$,
	\item $x,-x\in K$ implies $x=0$.
\end{enumerate}
Let \(K \subset \R^n\) be a  closed convex cone. The {\it polar cone} of \( K\) is the set  
\[
 K^\perp:=\{ x\in \R^n \mid  \langle x, y \rangle\leq 0, \forall \, y\in K\}.
\]
For any positive integer $p$ denote by $I_p$ the $p\times p$ identity matrix. Denote $I_m=I$ and $\diag (x)$ will denote a diagonal matrix corresponding to  elements of $x$.
For an $m\times m$ matrix $M$ consider the norm defined by 
$\|M\|:=\max_{x\ne0}\{\|Mx\|~:~\|x\|=1\}$,  this definition implies
\begin{equation} \label{eq:np}
\|Mx\|\leq \|M\|\|x\|, \qquad \|LM\|\le\|L\|\|M\|, 
\end{equation}
for any $m\times m$ matrices $L$ and $M$. 

The following lemma is Theorem 3.1.4 on page 45 of \cite{DennisSchnabe1996}.

\begin{lemma}[Banach's Lemma] \label{lem:ban}
Let $E$  be $m\times m$ matrix  and   $I$ the $m\times m$ identity matrix. If  $\|E\|<1$,  then $E-I$
is invertible and  $ \|(E-I)^{-1}\|\leq 1/\left(1-
\|E\|\right). $
\end{lemma}

Denote $\R^m_+=\{x=(x^1,\dots,x^m)\in\R^m:x_1\ge0,\dots,x^m\ge0\}$ the nonnegative orthant.
Let $A$ be an $m\times m$ nonsingular matrix. Then, the cone 
$$
K:=A\R^m_+, 
$$
 is called a  \emph{simplicial cone}.  Let $z\in \R^m$, then  the {\it projection $P_K(z)$ of the point $z$ onto the cone  \( K\)} is defined by 
$$
P_K(z):=\argmin \left\{\|z-y\|~:~ y\in K\right\}.
$$
\begin{remark} \label{re:ppo}
It easy to see that $P_{\R^m_+}(z)=z^{+}$ and it is well know that the projection onto a convex set  is continuous and nonexpansive, see \cite{HiriartLemarecal1}. 
\end{remark}
The above remark shows that  projection onto the  nonnegative orthant is an easy problem.  On the other hand, the projection onto a  general simplicial cone is 
difficult and computationally expensive, this problem has been studied e.g. in \cite{Frick1997,ujvari2007projection,NemethNemeth2009,EkartNemethNemeth2009}.  The statement
of the  problem that we are interested   is:
\begin{problem}[{\bf projection onto a simplicial cone}] \label{prob:pp}
 {\it   Given  $A$    an  $m\times m$ nonsingular matrix  and  $z\in \R^m$,  find  the  projection $P_K(z)$ of the point $z$ onto the  simplicial cone \( K=A\R^m_+\).}
\end{problem}
As we will see in the next section, by using Moreau's decomposition theorem for projecting onto cones,  solving Problem~\ref{prob:pp} is reduced to solving the following  
problem.
\begin{problem}[{\bf nonsmooth equation}]  \label{prob:ne}
 {\it   Given  $A$ an $m\times m$ nonsingular matrix   and  $z\in \R^m$,  find  the unique  solution $u$ of the nonsmooth equation}
\begin{equation}\label{equation}
		\left(A^\top A-I\right)x^++x=A^\top z.
\end{equation} 
In this case, $P_K(z)=Au^+$ where $K=A\R^m_+$.
\end{problem}
We will show in Section~\ref{sec:ssnm} that Problem~\ref{prob:ne} can   be solved by using a semi-smooth Newton method.

\section{Moreau's decomposition theorem for simplicial cones}
We recall the following result of Moreau \cite{Moreau1962}:
\begin{theorem} [Moreau's decomposition theorem]
Let $K,L\subseteq\R^m$ be two mutually polar cones in $\R^m$. Then, the following 
statements are equivalent:
\begin{enumerate}
	\item $z=x+y,~x\in K,~y\in L$ and $\langle x,y\rangle=0$,
	\item $x=P_K(z)$ and $y=P_L(z)$.
\end{enumerate}
\end{theorem}
\noindent
The following result follows from the definition  of the polar. For a proof see for example \cite{AbbasNemeth2012}.

\begin{lemma}\label{lad}  
Let $A$ be an $m\times m$ nonsingular matrix. Then,
\[
(A\mathbb{R}^m_+)^\perp=-(A^\top)^{-1}\mathbb{R}^m_+.
\]
\end{lemma}
The following result has been proved in \cite{AbbasNemeth2012} by using Moreau's decomposition theorem and Lemma~\ref{lad}.

\begin{lemma}\label{pm}
	Let $A$ be an $m\times m$ nonsingular matrix and $K=A\R^m_+$ the corresponding simplicial cone.
	Then, for any $z\in\R^m$ there exists a unique $x\in\R^m$ such that the following two 
	equivalent statements hold:
	\begin{enumerate}
		\item $z=Ax^+-(A^\top)^{-1}x^-,~x\in\R^m$,
		\item $Ax^+=P_K(z)$ and $-(A^\top)^{-1}x^-=P_{K^\perp}(z)$.
	\end{enumerate}
\end{lemma}
 The  following result is a direct consequence of Lemma \ref{pm}, it shows that solving Problem~\ref{prob:pp} is reduced to solving Problem~\ref{prob:ne}.
\begin{lemma}\label{pm2}
	Let $A$ be a nonsingular matrix, $K=A\R^m_+$ the corresponding simplicial cone
	and  $z\in\R^m$ arbitrary. Then, equation \eqref{equation} has a unique solution $u$ and $P_K(z)=Au^+$.
\end{lemma}
\begin{proof}
 Since $A$ is  an $m\times m$ nonsingular matrix, multiplying by $A^\top$, item (i) of Lemma \ref{pm} is  equivalently  transformed into 
$$
A^\top Ax^+ - x^-=A^\top z.
$$	
As $-x^-=x-x^+$, the above equation is equivalent to \eqref{equation}.  Therefore, equation \eqref{equation}  is equivalent to the equation in item (i) of Lemma \ref{pm}. 
Hence, we conclude from   Lemma \ref{pm} that equation \eqref{equation} has a unique solution $u$ and  $P_K(z)=Au^+$.
\end{proof}
\section{Semi-smooth Newton method} \label{sec:ssnm}
The {\it semi-smooth Newton iteration} for solving equation \eqref{equation} or equivalently for  finding the zero of the function 
$$
F(x):=\left(A^\top A-I\right)x^++x-A^\top z,  
$$ 
   is formally  defined by 
\begin{equation} \label{eq:nm}
F(x_k)+ S_k\left(x_{k+1}-x_{k}\right)=0,  \qquad   S_k \in \partial F(x_k), \qquad k=0,1,2,\ldots,
\end{equation}
where   $ \partial F(x_k)$  denotes the  Clarke generalized Jacobian of $F$ at $x_k$.  From the definition  of  Clarke generalized Jacobian of the function $F$ at $x$ 
(see  \cite{Clarke1990}) it easy to conclude that
$$
\left(A^\top A-I\right)\diag(\sgn(x^+) )+I\in  \partial F(x).
$$
Therefore, by using the last inclusion, the semi-smooth Newton iteration \eqref{eq:nm}   reduces to
\begin{equation*}
	\left(A^\top A-I\right)x_k^++x_k-A^\top z+\left(\left(A^\top A-I\right)\diag(\sgn(x_k^+))+I\right)(x_{k+1}-x_k)=0, 
\end{equation*}
and as $\diag(\sgn(x_k^+))x_k=x_k^+$, the  latter equality becomes
\begin{equation}\label{eq:mss}
	\left(\left(A^\top A-I\right)\diag(\sgn(x_k^+))+I\right)x_{k+1}=A^\top z,   \qquad k=0, 1, 2, \ldots , 
\end{equation}
which   formally  defines  a sequence  $\{x_k\}$ with starting point $x_0$,  called  the {\it semi-smooth Newton sequence} for solving equation \eqref{equation} or  for  projecting a point  $z\in \R^m$ onto the simplicial cone $K$.
\begin{lemma}\label{nonsing}
Let $A$  be an  $m\times m$ nonsingular matrix. Then the matrix 
\begin{equation} \label{eq;nma}
\left(A^\top A-I\right)\diag(\sgn(x^+))+I,
\end{equation}
is nonsingular for all $x\in \R^m$.  As a consequence, the semi-smooth Newton sequence $\{x_k\}$   is well defined from any starting point.
\end{lemma}
\begin{proof}
 To simplify the notations define  $B=A^\top A$.  Thus, the matrix in \eqref{eq;nma} becomes
$$
C:=(B-I)\diag(\sgn(x^+) )+I.  
$$
Hence, to prove the first part of the lemma we need to show that $C$ is a nonsingular matrix. 
	Let $x=(x^1,\dots,x^m)$ and the sets  
	$P(x)=\{i\in\{1,\dots,m\}:x^i>0\}$ 
	and $\tilde P(x)=\{1,\dots,m\} \backslash P(x)$. For  an $m\times m$ matrix $M$ define $M^{P(x)}$ and  $M^{\tilde P(x)}$  the matrices  defined, respectively,  by 
	\[
	M^{P(x)}_{ij}=\left\{
	\begin{array}{lll}
		M_{ij}&\textrm{ if }j\in P(x),\\
		0     &\textrm{ if }j\notin P(x).
	\end{array}
	\right. \qquad
	M^{\tilde P(x)}_{ij}=\left\{
	\begin{array}{lll}
		M_{ij}&\textrm{ if }j\in \tilde P(x),\\
		0     &\textrm{ if }j\notin \tilde P(x).
	\end{array}
	\right.
	\]
	Let $k:=m-|P(x)|$. Then, from the definitions of the matrix $C$ and the index sets $P(x)$ and ${\tilde P(x)}$, it is easy to see that 
	$$
	C=B^{P(x)}+I^{\tilde P(x)}, 
	$$ and
	there is an $m\times m$ permutation matrix $\Pi$ such that
	$\Pi^\top C \Pi$ is upper block triangular of the form
	\begin{equation} \label{eq:pm}
	\Pi^\top C \Pi=
	\left[\begin{array}{cc} I_k & E\\0 & F\end{array}\right],
	\end{equation}
	where $F$ is a principal submatrix of $B$. Since $A$ is  nonsingular,   $B=A^\top A$ is positive definite and $F$ is a principal submatrix of $B$,  we conclude  that the principal minor $\det F\ne0$ (see Corollary~ 7.1.5 of \cite{HornJohnson1985}).  Hence,  using \eqref{eq:pm} it follows that
	\[
	\det(\Pi^\top)\det C\det\Pi=\det(\Pi^\top C \Pi)=\det I_k\det F=
	\det F\ne0.
	\] 
	As $\det(\Pi^\top)=(\det\Pi)^{-1}$, the last relation implies $\det C=\det F\neq0$, hence the first part of the lemma is proven.
	
The proof of the second part of the lemma is immediate consequence of the definition of the sequence $\{x_k\}$ in \eqref{eq:mss} and the first part of the lemma.
\end{proof}

The next proposition give a condition for the Newton iteration \eqref{eq:mss} to finish in a finite
number of steps.
\begin{proposition}
	If in \eqref{eq:mss} it happens that $\sgn(x_{k+1}^+)=\sgn(x_k^+)$, then 
	$x_{k+1}$ solves equation \eqref{equation} and $P_K(z)=Ax_{k+1}^+$.
\end{proposition}
\begin{proof}
	If $\sgn(x_{k+1}^+)=\sgn(x_k^+)$ in equation \eqref{eq:mss}, then it becomes 
	\begin{equation}\label{ek}
		\left(\left(A^\top A-I\right)\diag(\sgn(x_{k+1}^+))+I\right)x_{k+1}=A^\top z.
	\end{equation}
	Since $\diag(\sgn(x_{k+1}^+))x_{k+1}=x_{k+1}^+$, \eqref{ek} yields
	\[
	\left(A^\top A-I\right)x_{k+1}^++x_{k+1}=A^\top z, 
	\] 
	which implies that  $x_{k+1}$ is a solution of   \eqref{equation} and, by using Lemma~\ref{pm2}, we have $P_K(z)=Ax_{k+1}^+$.
\end{proof}

The next proposition shows that the semi-smooth Newton sequence $\{x_k\}$ is bounded and gives a formula for any accumulation point of it.
\begin{proposition} \label{pr:bounded}
         The semi-smooth Newton sequence $\{x_k\}$ is bounded from any starting point. Moreover, for each accumulation point $\bar x$ of  $\{x_k\}$ there exists 
	 $\hat x \in \R^m$ such that $$\left(\left(A^\top A-I\right) \diag(\sgn(\hat{x}^+))+I\right){\bar x}=A^\top z.$$
\end{proposition}
\begin{proof}
 	First suppose that $\{x_k\}$ is unbounded. Then,  since $\{x_k\}$ is unbounded, the unit sphere is compact, and there are only finitely many vectors 
	$\sgn(x_{k}^+)$ with coordinates $0$ or $1$, it follows that there exists a vector $\tilde{x}\in\R^m$ and a subsequence $\{x_{k_j}\}$ of $\{x_k\}$ such that 
	\begin{equation} \label{eq:sssp}
	\lim_{j\to\infty}{\|x_{k_j+1}\|}=\infty,  \qquad \lim_{j\to\infty}\frac{x_{k_j+1}}{\|x_{k_j+1}\|}=v\ne0, \qquad\sgn(x_{k_j}^+)\equiv\sgn(\tilde{x}^+).
	\end{equation}
	Therefore, as $\sgn(x_{k_j}^+)= \sgn(\tilde{x}^+)$ for all $j$, the definition of the semi-smooth Newton sequence $\{x_k\}$ in 
	\eqref{eq:mss} implies
	$$
	\left(\left(A^\top A-I\right)\diag(\sgn(\tilde{x}^+))+I\right)\frac{x_{k_j+1}}{\|x_{k_j+1}\|}=\frac{A^\top z}{\|x_{k_j+1}\|},\qquad j=0,1,2,\ldots.
	$$
	By tending with $j$ to infinity in the above equality and by taking into account \eqref{eq:sssp}, it follows that 
	$$
	\left(\left(A^\top A-I\right)\diag(\sgn(\tilde{x}^+))+I\right)v=0, 
	$$
	which contradicts the first part of the Lemma \ref{nonsing} since $v\ne0$. Therefore, the sequence  $\{x_k\}$ is bounded, which proves the first part of the 
	proposition.   
	 
	For proving the second part of the lemma, let  $\bar x$ be an accumulation point of the sequence  $\{x_k\}$.  Then, since there are only finitely many vectors
	$\sgn(x_{k}^+)$ with coordinates  $0$ or $1$, there exists a vector $\hat{x}\in \R^m$ and a subsequence $\{x_{k_j}\}$ of $\{x_k\}$ such that 
	$$
	\lim_{j\to\infty}{x_{k_j+1}}=\bar x, \qquad   \sgn(x_{k_j}^+)\equiv  \sgn(\hat{x}^+),
	$$
	Since  $\sgn(x_{k_j}^+)=\sgn(\hat{x}^+)$  for all $j$, the definition of the semi-smooth Newton sequence $\{x_k\}$ in \eqref{eq:mss} implies	
	$$ 
	\left(\left(A^\top A-I\right) \diag(\sgn(\hat{x}^+))+I\right)x_{k_j+1}=A^\top z,\qquad j=0,1,2,\ldots.
	$$
	Taking the limit in the last equality as $k_j$ goes to $\infty$ ,  the second part of the proposition follows.
\end{proof}

The next lemma gives a condition for the convergence of the semi-smooth Newton sequence~$\{x_k\}$.
\begin{lemma}\label{ml}
	Let $a\in(0,1/2)$. If for any diagonal matrix $G$, with  diagonal elements $0$ or $1$  
	\begin{equation} \label{eq:cc}
	\left\|\left(\left(A^\top A-I\right)G+I\right)^{-1}\left(A^\top A-I\right)\right\|<a,
	\end{equation}
	 then the semi-smooth Newton sequence $\{x_k\}$ converges linearly to the unique solution $u$ of the equation   \eqref{equation} from any starting point and 
	 $P_K(z)=Au^+$.
\end{lemma}
\begin{proof}
	Let $u$ be the solution of equation \eqref{equation} and $\{x_k\}$ be the sequence defined in \eqref{eq:mss}. To simplify the notations, 
	denote 
	$$
	B=A^\top A, \qquad   D^k=\diag\left(\sgn(x_k^+)\right).
	$$
	Since $u$  is  the solution of equation \eqref{equation}, using the above notations and the definition of $\{x_k\}$ in \eqref{eq:mss}, we have
	 $$
	 (B-I)u^++u=A^\top z, \qquad \left(\left(B-I\right)D^k+I\right)x_{k+1}=A^\top z.
	 $$
	 By subtracting the two equalities above, and by taking into account that $D^k x_k=x_k^+$, we obtain, after some straightforward manipulations, that 
          $$
          (B-I)(x_k^+-u^+)+(B-I)D^k(x_{k+1}-x_k)+x_{k+1}-u=0,
          $$
	or equivalently, 
	$$
	\left((B-I)D^k+I\right)(x_{k+1}-u)+(B-I)\left(x_k^+-u^++D^k(u-x_k)\right)=0,
	$$
	Since $D^k=\diag\left(\sgn(x_k^+)\right)$, by using the last equality and Lemma~\ref{nonsing}, we get
	$$
	x_{k+1}-u=-\left((B-I)D^k+I\right)^{-1}(B-I)\left(D^k(x_k-u)+(x_k^+-u^+)\right).
	$$
	By using the first inequality in \eqref{eq:np}, the Cauchy inequality,  $\|D^k\|\leq 1$, and the nonexpansivity of the projection mapping 
	$x\mapsto x^+$ onto the nonnegative orthant (see Remark~\ref{re:ppo}), the latter relation and assumption \eqref{eq:cc} gives
	$$
	\|x_{k+1}-u\|\le \left\|\left((B-I)D^k+I\right)^{-1}(B-I)\right\|2\|x_k-u\|<2a\|x_k-u\|,
	$$
	where $2a<1$. As the last inequality holds for all $k$,  we conclude that the sequence $\{x_k\}$ converges linearly to $u$ and Lemma~\ref{pm2} implies
	$P_K(z)=Au^+$.
\end{proof}

The next theorem provides a sufficient condition for the linear convergence of the Newton 
iteration. 
\begin{theorem}
Let $b\in(0,1/3)$. If 
\begin{equation} \label{eq:bl13}
 \|A^\top A-I\|<b,
\end{equation}
then the semi-smooth Newton sequence $\{x_k\}$  converges linearly to the unique solution $u$ of equation \eqref{equation} from any starting point  and $P_K(z)=Au^+$.
\end{theorem}
\begin{proof}
	To simplify the notations,  denote $B=A^\top A$.
	By Lemma $\ref{ml}$ it is enough to show that \[\|((B-I)G+I)^{-1}(B-I)\|<\f{b}{1-b}<\f12,\] 
	for any diagonal matrix $G$ with diagonal elements $0$ or $1$. By using the properties \eqref{eq:np} of the norm, it is enough to show
	that 
	\begin{equation}\label{show}
		\|(I-(I-B)G)^{-1}\|\|(I-B)\|<\f{b}{1-b}.
	\end{equation} 
	Since $\|G\|\leq 1$, assumption \eqref{eq:bl13} and the properties \eqref{eq:np} of the norm imply 
	$$
	\|(I-B)G\|\leq\|I-B\|\|G\|\le\|I-B\|<b<1.
	$$
	Thus, by applying Lemma~\ref{lem:ban} with $E=(I-B)G$ and by using the last inequality, we conclude that
	$$
		\|(I-(I-B)G)^{-1}\|\|I-B\|\leq \ds\f{\|I-B\|}{1-\|(I-B)G\|}<\frac{b}{1-b},
	$$
	which is exactly the required relation \eqref{show}. By using Lemma~\ref{pm2}, we have $P_K(z)=Au^+$.
\end{proof}

\section{Conclusions}
In this paper we studied  the problem of   projection onto a  simplicial cone  which,  via  Moreau's decomposition theorem for projecting onto cones,  is reduced to 
finding the unique solution of a  nonsmooth system of equations.  Our main result shows that,   under a mild assumption on the simplicial cone,   we can  apply  a
semi-smooth Newton method for finding a unique solution of  the  obtained  associated system and  that  the generated sequence converges linearly to the  solution   for 
any starting point.   It would be interesting to see whether the used technique can be applied  for finding  the projection onto  more general cones.   As has been shown 
in \cite{ujvari2007projection}, the problem of projection onto a   simplicial  cone is reduced to a certain type of linear complementarity problem (LCP).  Then, another 
interesting problem to address is to compare the four  methods  for projecting onto a  simplicial  cone, namely,  semi-smooth Newton method,  the  methods proposed in 
\cite{ujvari2007projection,Morillas2005,EkartNemethNemeth2009} and  the Lemke's method for LCPs.

\bibliographystyle{habbrv}
\bibliography{simcoproj}

\end{document}